\documentclass[a4paper,twoside]{amsart}
\usepackage{amsmath}
\usepackage{amsthm}
\usepackage{amsfonts}
\usepackage{amssymb}
\usepackage{mathrsfs}
\usepackage{enumerate}
\usepackage{mathtools}
\usepackage{tikz}
\usetikzlibrary{arrows}
\usepackage[final=true,colorlinks=true,linkcolor=black,urlcolor=black,bookmarksopen=true]{hyperref}

\theoremstyle{plain}
\newtheorem{lemma}{Lemma}[section]
\newtheorem{prop}[lemma]{Proposition}
\newtheorem{theo}[lemma]{Theorem}
\newtheorem{coro}[lemma]{Corollary}
\theoremstyle{remark}
\newtheorem{rem}[lemma]{Remark}
\theoremstyle{definition}
\newtheorem{definition}[lemma]{Definition}
\newtheorem{ex}[lemma]{Example}

\DeclareMathOperator{\obj}{Obj}
\DeclareMathOperator{\mor}{Mor}
\DeclareMathOperator{\aut}{Aut}
\DeclareMathOperator{\im}{Im}
\DeclareMathOperator{\Hom}{Hom}
\DeclareMathOperator{\coker}{\textnormal{coker}}
\DeclareMathOperator{\sop}{\textnormal{sop}}
\DeclareMathOperator{\grpd}{\textnormal{\bf Grpd}}
\DeclareMathOperator{\deck}{\textnormal{\bf Deck}}

\newcommand{\op}{\textnormal{op}}
\newcommand{\id}{\mathrm{Id}}
\newcommand{\C}{\mathscr{C}}
\newcommand{\G}{\mathcal{G}}
\newcommand{\n}{\mathcal{N}}
\newcommand{\Ds}{\mathscr{D}}
\newcommand{\N}{\mathbb{N}}
\newcommand{\Z}{\mathbb{Z}}

\newcommand{\K}{\mathcal{K}}
\newcommand{\loc}{\mathcal{L}}
\newcommand{\set}{\textnormal{\bf Set}}

\newcommand{\Top}{\textnormal{\bf Top}}
\newcommand{\ATop}{\textnormal{\bf ATop}}
\newcommand{\pos}{\textnormal{\bf Pos}}
\newcommand{\ord}{\textnormal{\bf Ord}}
\newcommand{\cov}{\textnormal{\bf Cov}}
\newcommand{\cat}{\textnormal{\bf Cat}}
\newcommand{\scpx}{\textnormal{\bf SCpx}}
\newcommand{\grp}{\textnormal{\bf Grp}}

\title{Regular coverings and fundamental groupoids of Alexandroff spaces}

\author{Nicol\'as Cianci}
\address{Facultad de Ciencias Exactas y Naturales \\ Universidad Nacional de Cuyo \\ Mendoza, Argentina.}
\email{nicocian@gmail.com}



\thanks{Research partially supported by grant M044 of SeCTyP, UNCuyo. The author was also partially supported by a CONICET doctoral fellowship.}

\begin{document}
	\begin{abstract}
		We summarize several results about the regular coverings and the fundamental groupoids of Alexandroff spaces. In particular, we show that the fundamental groupoid of an Alexandroff space $X$ is naturally isomorphic to the localization, at its set of morphisms, of the thin category associated to the set $X$ considered as a preordered set with the specialization preorder. We also show that the regular coverings of an Alexandroff space $X$ are represented by certain morphism-inverting functors with domain $X$, extending a result of E. Minian and J. Barmak about the regular coverings of locally finite T$_0$ spaces. 
	\end{abstract}
	\maketitle

	\section{Introduction}
	Alexandroff spaces are topological spaces in which arbitrary intersections of open sets are open. These spaces, which are tipically not found among the most well studied spaces in algebraic topology, were introduced and studied by P. Alexandroff in his 1937 article \cite{alexandroff1937diskrete}.
	
	Alexandroff constructed a bijective correspondence between T$_0$ Alexandroff spaces and partially ordered sets. This correspondence is, in fact, an isomorphism between the category $\ATop_0$ of T$_0$ Alexandroff spaces and continuous maps and the category $\pos$ of partially ordered sets and order-preserving morphisms. It is not hard to see that this isomorphism extends to an isomorphism between the category $\ATop$ of Alexandroff spaces and continuous maps and the category $\ord$ of preordered sets and order-preserving morphisms.
	
	Alexandroff's work provided a basis for the later works of R. E. Stong \cite{stong1966finite} and M. McCord \cite{mccord1966singular} about finite topological spaces which constitute an important class of Alexandroff spaces. 
	
	Stong constructed a canonical form for the homotopy type of a finite topological space $X$, called the \emph{core} of $X$, which is a T$_0$ strong deformation retract of $X$ with no \emph{beat points}\footnote{This terminology is due to J. P. May \cite{may2003fts}.} and was able to prove that
	\begin{enumerate}
		\item every finite topological space has a core which is unique up to homeomorphism, and
		\item two finite topological spaces are homotopically equivalent if and only if their cores are homeomorphic. 
	\end{enumerate}
	Moreover, a core of a finite topological space can be constructed by performing a simple algorithm and hence, the problem of determining whether two finite topological spaces are homotopically equivalent can be easily reduced to the problem of determining whether two finite posets are isomorphic.
	
	Stong's results about the classification of homotopy types of finite spaces through the study of cores were later extended to larger classes of Alexandroff spaces \cite{arenas1999alexandroff,kukiela2010homotopy,chen2015cores}.
	
	On the other hand, M. McCord studied weak homotopy types of Alexandroff spaces and showed that 
	\begin{enumerate}
		\item for every Alexandroff space $X$, there exists a simplicial complex $K$ and a weak homotopy equivalence $f_X\colon |K|\to X$, and
		\item for every simplicial complex $K$ there exists an Alexandroff space $X$ and a weak homotopy equivalence $f_K\colon |K|\to X$,
	\end{enumerate}
	where $|\cdot|$ is the usual \emph{geometric realization} functor from the category of simplicial complexes and simplicial maps to the category of topological spaces and continuous maps. 
	Hence, McCord's results show that the problem of computing the homotopy, homology and cohomology groups of Alexandroff spaces is equivalent to the problem of computing the homotopy, homology and cohomology groups of simplicial complexes. In particular, while the computation of the singular homology groups of finite topological spaces can be performed by standard methods (see \cite[Section 2]{cianci2017new}), the example at the end of \cite{mccord1966singular} shows that there exists a space with six points which is weakly equivalent to the sphere $S^{2}$ and hence its homotopy groups of high degree remain unknown. 
	
	As a consequence of McCord's results, we see that maps from spheres to Alexandroff spaces are as ``interesting'' as maps from spheres to polyhedra, for most practical purposes within the scope of the theory of weak homotopy types. On the other hand, since the category $\ord$ can be regarded as a full subcategory of $\cat$, the isomorphism $\ATop\cong\ord$ allows us to consider functors from Alexandroff spaces to other categories, such as groupoids. This is the main idea  
	behind the article \cite{barmak2012colorings} of J. Barmak and E. Minian, where the fundamental group and the regular coverings of \emph{locally finite} T$_0$ spaces are studied by means of combinatorial methods.
	Following the ideas of \cite{barmak2012colorings}, we show that the canonical inclusion $\iota_X\colon X\to \loc X$ of an Alexandroff space, regarded as a preordered set, into its localization induces an isomorphism of fundamental groups for every base point. This result allows us to study the fundamental groupoids and the regular coverings of Alexandroff spaces by means of classic combinatorial tools from Category Theory.
	
	The main purpose of this article is to summarize several results about fundamental groupoids and regular coverings of Alexandroff spaces for future reference.
	Most of these are original. Some of the results of subsection \ref{subs:coverings} are mentioned in \cite{cianci2018splitting} and were included both for the sake of completeness and to provide several omitted details in their original formulations and proofs. Other results were already known and have been included in this article either because detailed proofs could not be found in the literature or because we wished to provide alternative proofs of our own. 
	
	\section{Preliminaries and notation}
	Throughout this article, we will write:
	\begin{itemize}
		\item $\set$ for the category of sets and functions.
		\item $\cat$ for the category of small categories and functors.
		\item $\cat_*$ for the category of pointed objects in $\cat$, that is, the category of small categories with basepoint and basepoint-preserving functors.
		\item $\Top$ for the category of topological spaces and continuous functions.
		\item $\grpd$ for the category of groupoids and groupoid homomorphisms.
		\item $\grp$ for the category of groups and group homomorphisms.
		\item $\ord$ for the category of preordered sets and order preserving morphisms.
		\item $\pos$ for the category of ordered sets and order preserving morphisms.
		\item $\Ds^{\C}$ for the category of functors from $\C$ to $\Ds$ and natural transformations, for categories $\C$ and $\Ds$. 
	\end{itemize}
	Since preordered sets are equivalent to \emph{thin categories}, we will regard $\ord$ and $\pos$ as full subcategories of $\cat$ in the usual fashion\footnote{That is, by considering a preordered set $(X,\leq)$ as a small category with $X$ as its set of objects, and with a unique arrow from $x$ to $y$, which is usually denoted by $x\to y$ (or simply $x\leq y$), if and only if $x\leq y$, for every $x,y\in X$.}. On the other hand, since groups are equivalent to one-object groupoids, we will regard $\grp$ as a full subcategory of $\grpd$ in the usual fashion	as well. Thus we have the inclusions of categories
	\[
		\pos\hookrightarrow \ord\hookrightarrow\cat
	\]
	and
	\[
	\grp\hookrightarrow \grpd\hookrightarrow\cat.
	\]
	
	For a category $\C$ and an object $c_0$ of $\C$, the group of automorphisms of $c_0$ in $\C$ will be denoted by $\aut_{\C}(c_0)$. Note that, when considered as a category, the group $\aut_{\C}(c_0)$ is nothing but the subcategory of $\C$ with object $c_0$ and arrows the automorphisms of $c_0$. Hence, we will have an inclusion functor 
	\[\aut_{\C}(c_0)\hookrightarrow \C.\]
	\subsection{Alexandroff spaces}
	\begin{definition}
		A topological space $X$ is called an \emph{Alexandroff space} (or \emph{A-space} for short) if its topology is closed under (arbitrary) intersections, that is, the intersection of open sets of $X$ is again open in $X$.
		
		The category of Alexandroff spaces and continuous functions (that is, the full subcategory of $\Top$ whose objects are the Alexandroff spaces) will be denoted by $\ATop$. The full subcategory of $\ATop$ whose objects are the T$_0$ Alexandroff spaces will be denoted by $\ATop_0$. We have the following inclusions of categories:
		\[
		\ATop_0\hookrightarrow \ATop\hookrightarrow\Top.
		\]
	\end{definition}
	\begin{definition}
		Let $X$ be an Alexandroff space and let $x\in X$. The \emph{minimal open set} $U_x$ of $x$ is defined as the intersection of every open set in $X$ that contains $x$. It is immediate that $U_x$ is the minimal open neighbourhood of $x$ with respect to set inclusion.
	\end{definition}

	\begin{definition}\label{def:spec_order}
		Let $X$ be an Alexandroff space. We define the \emph{specialization preorder} in $X$ as the preorder $\leq_X$ (or $\leq$ if the space $X$ is understood) defined by
		\[
		x\leq_X y\quad\text{if and only if} \quad U_x\subseteq U_y
		\]
		for every $x,y\in X$. Equivalently, 
		\[
		x\leq_X y\quad\text{if and only if} \quad x\in U_y
		\]
		for every $x,y\in X$.
		
		We will regard $X$ as a preordered set with the specialization preorder without further notice.
		Note that a subset of $X$ is open if and only if it is a lower set, and that it is a closed set if and only if it is an upper set. 
		
		It is easy to see that $\leq_X$ is an order if and only if $X$ is a T$_0$ space.
	\end{definition}

	It is well known (and not hard to prove) that the construction described in \ref{def:spec_order} defines an isomorphism between the categories $\ATop$ and $\ord$ which clearly restricts to an isomorphism between $\ATop_0$ and $\pos$.
	
	In \cite{mccord1966singular}, M. McCord shows that for every Alexandroff T$_0$ space $X$ there exists a simplicial complex $\K(X)$ and a weak homotopy equivalence $\varphi_X\colon |\K(X)|\to X$, where $|\K(X)|$ denotes the \emph{geometric realization} of $\K(X)$ (the definition and properties of the \emph{geometric realization functor} $|\cdot|\colon \scpx\to\Top$ from the category $\scpx$ of simplicial complexes to the category of topological spaces can be found in \cite[Chapter~3]{spanier1981algebraic}). The simplicial complex $\K(X)$ has the elements of $X$ as vertices and the finite non-empty chains of $X$ as simplices, and the function $\varphi_X$ maps every element of $|\K(X)|$ to the minimum of its support. It is not hard to see that the weak equivalence $\varphi_X$ is natural in $X$: for every continuous map $f\colon X\to Y$ between Alexandroff T$_0$ spaces, one has that $f\varphi_X=\varphi_Y|\K(f)|$ where $\K(f)\colon \K(X)\to \K(Y)$ is the obvious simplicial map induced by $f$.
	In other words, $\varphi$ is a natural weak equivalence form the functor
	\[\ATop_0\xrightarrow{|\K|}\Top\]
	to the inclusion funtor
	\[\ATop_0\hookrightarrow\Top.\]
	
	In order to extend this result to the category of Alexandroff spaces, we will need the next theorem, which is 
	well known and can be found, for example, in \cite{quillen1973higher,raptis2010homotopy}. Since we have not been able to find a proof in the bibliography, we decided to include an elementary yet detailed proof of our own.
	
	Before stating the aforementioned result, we need to fix some notation and recall some basic facts about the \emph{classifying space of small categories} \cite{segal1968classifying}.
	
	For $n\in\N_0$ we will write $[n]$ for the poset $\{0,1,\ldots,n\}$ with the usual order and $\Delta^{n}$ for the standard topological $n$-simplex
	\[\Delta^{n}=\{(t_0,t_1,\ldots,t_n)\in I^{n+1}:\sum\limits_{k=0}^{n}t_k=1\}\] with the obvious subspace topology, where $I$ denotes the unit interval $[0,1]$ with the usual topology.
	
	Recall that the classifying space $B\C$ of a small category $\C$ is defined as the geometric realization of the simplicial nerve of $\C$. Explicitely, $B\C$ is the quotient $\overline{N\C}/\sim$ where
	\begin{itemize}
		\item $N\C$ denotes the simplicial nerve of $\C$, that is, the simplicial set defined by $N\C_n=\C^{[n]}$ for every $n\in\N_0$,
		\item the space $\overline{N\C}$ is defined as the coproduct
		\[\overline{N\C}=\coprod\limits_{n=0}^{\infty}N\C_n\times \Delta^{n},\] and
		\item The relation $\sim$ is the equivalence relation in $\overline{N\C}$ generated by 
		\[(d_i \omega,t)\sim (\omega,\delta_i t), \text{ for $\omega\in\C^{[n]}$ and $t\in\Delta^{n-1}$, with $n\in\N$ and $i\in\{0,\ldots,n\}$,}\]
		and
		\[(s_i \omega,t)\sim (\omega,\sigma_i t), \text{ for $\omega\in\C^{[n]}$ and $t\in\Delta^{n+1}$, with $n\in\N_0$ and $i\in\{0,\ldots,n\}$,}\]
		where $d_i$ and $\delta_i$ are the usual face maps and $s_i$ and $\sigma_i$ are the usual degeneracy maps.
	\end{itemize}
	We will write $[\omega,t]$ for the equivalence class of $(\omega,t)$ in $B\C$ for every $\omega\in\C^{[n]}$ and every $t\in\Delta^{n}$.
	
	The construction $B$ extends to a functor $B\colon\cat\to\Top$ by defining $BF\colon B\C\to B\Ds$ as $BF([\omega,t])=[F\circ \omega,t]$ for every functor $F\colon \C\to\Ds$ between small categories. Moreover, a natural transformation $\alpha\colon F\Rightarrow G\colon \C\to\Ds$ induces a homotopy $\alpha_*\colon BF\simeq BG$ \cite[Proposition 2.1]{segal1968classifying}.
	
	\begin{theo}\label{theo:B=|K|}
		The functors
		\[\pos\xrightarrow{\K}\scpx\xrightarrow{|\cdot|}\Top\]
		and
		\[\pos\hookrightarrow\cat\xrightarrow{B}\Top\] 
		are naturally isomorphic.
	\end{theo}
	\begin{proof}

		Let $X$ be a poset. In order to simplify the notation, we will denote the characteristic function $\chi_x\colon X\to I$  by $x$ for every $x\in X$. Given $t\in\Delta^{n}$, the $i$th coordinate of $t$ will be denoted by $t_i$ for $i=0,\ldots,n$.
		
		For each $m$--simplex $r$ of $\K(X)$, we will write $l_r$ for the composition
		\[\Delta^{m}\cong |r|\hookrightarrow |\K(X)|.\]
		Now, for each $n$--simplex $\omega$ of $N(X)$, the image of $\omega$ (when considered as an order preserving morphism from $[n]$ to $X$), is a simplex of $\K(X)$ that will be denoted by $r_\omega$. It is clear that if $m=\dim r_\omega$ then the factorization of $\omega$ through $r_\omega$ induces an order preserving morphism $[n]\to [m]$ that induces a continuous map 
		\[\omega_*\colon \Delta^{n}\to\Delta^{m}\] defined by 
		\[\omega_*(t)_k=\sum\limits_{\omega(j)=a_k}t_j\] for every $k=0,\ldots,m$, where $r_\omega=\{a_1,\ldots,a_m\}$ with $a_1<\ldots<a_m$. Note that $\omega_*=\id_{\Delta^{n}}$ if $\omega$ is a non-degenerated simplex.
		
		On the other hand, note that since $N(X)_n$ is a set which is regarded as a discrete topological space for each $n\in \N$, then the space $N(X)_n\times \Delta^{n}$ can be identified with the coproduct 
		\[\coprod\limits_{\omega\in N(X)_n}\Delta^{n}.\]
		It follows that the space $\overline{N(X)}$ is the coproduct
		\[\coprod\limits_{n\in\N_0}\coprod\limits_{\omega\in X^{[n]}}\Delta^{n}.\]
		
		Hence, the family of functions $\{l_{r_\omega}\omega_*:\text{$\omega$ is a simplex of $N(X)$}\}$ induces a continuous map
		\[f\colon \overline{N(X)}\to |\K(X)|\] clearly defined by $f(\omega,t)=\sum\limits_{k=0}^{n}t_k\omega(k)$ for each $(\omega,t)\in N(X)_n\times\Delta^{n}$, for each $n\in\N_0$. 
		
		It is not hard to see that $f$ is a quotient map. On the other hand, an explicit computation shows that 
		\[f(d_i\omega,t)=f(\omega,\delta_i t)\] for each $\omega\in N(X)_n$ and each $t\in\Delta^{n-1}$, for $n\in\N$ and $0\leq i\leq n$, and that 
		\[f(s_i\omega,t)=f(\omega,\sigma_i t)\] for each $\omega\in N(X)_n$ and each $t\in\Delta^{n+1}$, for $n\in\N_0$ and $0\leq i\leq n$.
		
		We wish to show that any two elements of $\overline{N(X)}$ that maps to the same function through $f$ are equivalent. Before, we will see that every $(\omega,t)$ of $\overline{N(X)}$ is equivalent to $(\overline{\omega},\overline{t})$ where $\overline{\omega}$ is the non-degenerated simplex that maps bijectively on the support of $f(\omega,t)$ and 
		\[\overline{t}_j=f(\omega,t)(\overline{\omega}(j))=\sum\limits_{\omega(k)=\overline{\omega}(j)}t_k\]
		for every $j=0,\ldots,\dim \overline{\omega}$ (cf. \cite[Lemma 3]{milnor1957geometric}).
		
		First, note that if $\omega$ is degenerated, then there exists a (unique) non-degenerated simplex $\omega'$ and degenerations $s_{i_1},\ldots,s_{i_p}$ such that $\omega=s_{i_p}\ldots s_{i_1}\omega'$. Letting $t'=\sigma_{i_1}\ldots \sigma_{i_p}t$, it is clear that
		\[(\omega,t)\sim(\omega',t').\]
		If $\omega$ is non-degenerated, we let $\omega'=\omega$ and $t'=t$. Since $(\omega,t)\sim(\omega',t')$, it follows that $f(\omega,t)=f(\omega',t')$ and hence 
		\[t'_j=f(\omega',t')(\omega'(j))=f(\omega,t)(\omega'(j))=\sum\limits_{\omega(k)=\omega'(j)}t_k\] for $j\in[\dim \omega']$.
		
		Now, if one or more coordinates of $t'$ are equal to zero, then there exist $q\in\N_0$, $\overline{t}\in\Delta^{q}$ and face operators $\delta_{j_1},\ldots,\delta_{j_{p'}}$ such that $\overline{t}$ has all its coordinates different from zero and $t'=\delta_{j_{p'}}\ldots\delta_{j_1}\overline{t}$. Letting $\overline{\omega}=d_{j_1}\ldots d_{j_{p'}}\omega'$, we see that
		\[(\omega',t')\sim (\overline{\omega},\overline{t}).\]
		If every coordinate of $t'$ if different from zero, we let $\overline{\omega}=\omega'$ and $\overline{t}=t'$.
		
		Note that, since $X$ is a poset and $\omega'$ is non-degenerated, $\overline{\omega}$ is non-degenerated as well and for each $j\in [\dim\overline{\omega}]$ there exists a unique $k\in [\dim \omega']$ such that $\omega'(k)=\overline{\omega}(j)$. In this case, we obtain that
		\[\overline{t}_j=f(\overline{\omega},\overline{t})(\overline{\omega}(j))=f(\omega',t')(\omega'(k))=t'_k.\]
		
		It is clear that $(\omega,t)\sim (\overline{\omega},\overline{t})$. It can be easily seen that the image of $\overline{\omega}$ is the support of $f(\omega,t)$ and that 
		\[\overline{t}_j=f(\overline{\omega},\overline{t})(\overline{\omega}(j))=f(\omega,t)(\overline{\omega}(j))=\sum\limits_{\omega(h)=\overline{\omega}(j)}t_h\]
		for each $j=0,\ldots,\dim \overline{\omega}$ as we wished to show.
		
		Now, let $(\omega,t)$ and $(\omega',t')$ elements of $\overline{N(X)}$ such that $f(\omega,t)=f(\omega',t')$.
		Then 
		\[\sum\limits_{h=0}^{\dim\omega}t_h\omega(h)=\sum\limits_{j=0}^{\dim\omega'}t'_j\omega'(j).\]
		It is clear that \[\im\overline{\omega}=\sop f(\omega,t)=\sop f(\omega',t')=\im\overline{\omega'}\] and that 
		\[\overline{t}_j=f(\omega,t)(\overline{\omega}(j))=f(\omega',t')(\overline{\omega'}(j))=\overline{t'}_j\] whenever $0\leq j\leq\dim\overline{\omega}$. 
		Thus,
		\[(\omega,t)\sim (\overline{\omega},\overline{t})=(\overline{\omega'},\overline{t'})\sim (\omega',t').\]
		Hence, the function \[f_X\colon BX=\left(\overline{N(X)}/\sim\right)\to |\K(X)|\] defined by $f_X([\omega,t])=f(\omega,t)$ for each $[\omega,t]\in BX$ is a homeomorphism.
		
		Moreover, given posets $X$ and $Y$ and an order preserving morphism $g\colon X\to Y$, an explicit computation shows that 
		\[(|\K(g)|\circ f_X)([\omega,t])(y)=\sum\limits_{k\in(g\omega)^{-1}(y)}t_k=(f_Y\circ Bg)([\omega,t])(y)\]
		for every $[\omega,t]\in BX$ and every $y\in Y$. It follows that
		\[|\K(g)|f_X=f_Y Bg.\]
		
		The result follows.
	\end{proof}
	\begin{definition}
		\label{def:varphi_X}
		For every $t\in \Delta^{n}$ we write $m_t$ for the minimum of the set 
		\[\{i\in\{0,\ldots,n\}:t_i\neq 0\}.\]
		
		Let $X$ be an Alexandroff space. 
		We define a function $\varphi_X\colon BX\to X$ by $\varphi_X([\omega,t])=\omega(m_t)$ for every $[\omega,t]\in BX$.
		It is not hard to see that $\varphi_X$ is well defined.
		
		Moreover, it is easy to see that $f\varphi_X=\varphi_Y Bf$ for every continuous map $f\colon X\to Y$ between Alexandroff spaces $X$ and $Y$.
		
		Note that if $X$ is an Alexandroff T$_0$ space, then $\varphi_X$ is nothing but the weak equivalence obtained by combining Theorem \ref{theo:B=|K|} and \cite[Theorem 1]{mccord1966singular}.
	\end{definition}

	\begin{theo}\label{theo:varphi_X}
		The collection of maps $\{\varphi_X:X\text{ is an Alexandroff space}\}$ defines a natural weak equivalence $\varphi$ from the funtor
		\[
		\ATop\cong\ord\hookrightarrow\cat\xrightarrow{B}\Top
		\]
		to the inclusion functor
		\[
		\ATop\hookrightarrow\Top.
		\]
	\end{theo}
	\begin{proof}

		From Theorem \ref{theo:B=|K|} and \cite[Theorem 1]{mccord1966singular} we know that $\varphi_X$ is a weak homotopy equivalence for every Alexandroff T$_0$ space.
		We need to show that $\varphi_X$ is a weak homotopy equivalence for every Alexandroff (not necessarily T$_0$) space $X$.
		Let $X$ be an Alexandroff space, let $X_0$ be its Kolmogorov quotient\footnote{The Kolmogorov quotient of a topological space $X$ is the quotient $X/\sim$ where $\sim$ is the equivalence relation that identifies topologically indistinguishable points (points that have the exact same heighbourhoods). The space $X/\sim$ is a T$_0$ space and if $X$ is a T$_0$ space, then the quotient map $q_X\colon X\to X/\sim$ is a homeomorphism.}, let $q_X\colon X\to X_0$ be the associated quotient map and let $i_X$ be any section of $q_X$. It is straightforward to see that $q_X$ is an equivalence of categories\footnote{In fact, $X_0$ is nothing but the skeleton of $X$ which is constructed by identifying the isomorphic objects of $X$, and $q_X$ and $i_X$ are the usual ``quotient'' and ``inclusion'' functors, respectively.} with inverse $i_X$. It follows that $Bq_X$ is a homotopy equivalence\footnote{
		M. McCord shows in \cite[Lemma 9]{mccord1966singular} that $q_X$ is a homotopy equivalence in the usual (topological) sense.}.
	
		Note that $q_X\varphi_X$ is continuous since it is equal to $\varphi_{X_0}Bq_X$. Continuity of $\varphi_X$ follows easily.
		
		By the 2-out-of-3 property of weak homotopy equivalences, it follows that $\varphi_X$ is a weak homotopy equivalence as we wanted to show.
	\end{proof}
	\begin{coro}
		Let $f\colon X\to Y$ be a continuous map between Alexandroff spaces. Then, $f$ is a weak homotopy equivalence if and only if $Bf\colon BX\to BY$ is an homotopy equivalence.
	\end{coro}
	\begin{proof}
		From Theorem \ref{theo:varphi_X} and the 2-out-of-3 property of weak homotopy equivalences it follows that the map $f$ is a weak homotopy equivalence if and only if $Bf$ is a weak homotopy equivalence. Since $BX$ and $BY$ are CW-complexes, the result follows immediately from Whitehead's Theorem.
	\end{proof}
	\begin{rem}
		It is well known that Whitehead's Theorem does not hold in the context of Alexandroff spaces, that is, that there exist weak homotopy equivalences between Alexandroff spaces that are not homotopy equivalences. In fact, there exist weakly contractible\footnote{A topological space is weakly contractible if all its homotopy groups are trivial.} finite topological  spaces which are not contractible. As we show in \cite{cianci2016smallest}, the minimum cardinality of such a space is 9, and there exist, up to homeomorphism, two weakly contractible non-contractible spaces of 9 points.
	\end{rem}

	\subsection{Groupoids and localization}
	Recall that for every small category $\C$ there exists a category $\loc\C$, which will be called a \emph{localization of $\C$}, and a functor $\iota_{\C}\colon \C\to\loc\C$ such that:
	\begin{enumerate}
		\item $\iota_\C$ is a morphism-inverting functor (that is, $\iota_\C(f)$ is an isomorphism for every arrow $f$ in $\C$) and,
		\item for every morphism-inverting functor $F\colon \C\to \Ds$, there exists a unique functor $\overline{F}\colon \loc\C\to\Ds$ such that $F=\overline{F}\iota_{\C}$.
	\end{enumerate}
  The principal reference for localization of categories is \cite{gabriel1967calculus}.
	
	It is not hard to see that localizations of $\C$ are unique up to a unique canonical isomorphism for every small category $\C$, and thus it is usual to consider any localization of $\C$ as ``the'' localization of $\C$. 
	
	\begin{ex}
		Let $\G$ be a grupoid. It is clear that $\G$ together with the functor $\id_\G$ is a localization of $\G$. Thus, we may always assume that $\loc\G=\G$ and that $\iota_{\G}=\id_{\G}$.
	\end{ex}
	
	For a small category $\C$, we can define a localization $\loc\C$ of $\C$ as follows. 
	First, we consider the set $\Omega=\mor(\C)\sqcup\mor(\C)$ and the canonical inclusions $i_0,i_1\colon \mor\C\to\Omega$. The source and target maps $s,t\colon \mor(\C)\to\obj(\C)$ determine unique maps $S,T\colon\Omega\to\obj(\C)$ such that $Si_0=s$, $Si_1=t$, $Ti_0=t$ and $Ti_1=s$. We will tacitly identify $\mor(\C)$ with its image through $i_0$ and hence we will denote $i_0(f)$ by $f$ for every $f\in\mor(\C)$. Under this identification, we have that $\Omega=\mor(\C)\cup \mor(\C)^{-1}$, where $\mor(\C)^{-1}=i_1(\mor(\C))$. The map $i_1(f)$ will be denoted by $\overline{f}$ for every $f\in\mor(\C)$.
	
	Let $c,c'\in\obj(\C)$. By a \emph{path in $\C$ from $c$ to $c'$} we will mean a finite sequence $(c_0,f_0,c_1\ldots,c_{n},f_n,c_{n+1})$ where 
	\begin{itemize}
		\item $c_i\in \obj(\C)$ for every $i=0,\ldots,n+1$,
		\item $f_i\in\Omega$ for every $i=0,\ldots,n$,
		\item $c_i=S(f_i)$ and $c_{i+1}=T(f_i)$ for every $i=0,\ldots,n$,
		\item $c_0=c$, and
		\item $c_{n+1}=c'$.
	\end{itemize}
	We will explicitly allow $n$ to be $-1$ in our definition and thus the sequence $(c)$ will be a path from $c$ to itself.
	The path $(c_0,f_0,c_1\ldots,c_{n},f_n,c_{n+1})$ will be denoted simply by $(f_0,\ldots,f_n)$ while the path $(c)$ will be denoted by $()_c$.
	
	Now, we can define a category $F(\C)$ with the same objects as $\C$ by letting $\Hom_{F(\C)}(c,c')$ be the set of paths in $\C$ from $c$ to $c'$, for $c,c'\in\obj(\C)$. The composition in $F(\C)$ is defined by concatenation, that is, the path $()_c$ is the identity at $c$ for every $c\in \obj(\C)$ and
	\[(f_0,\ldots,f_n)(g_0,\ldots,g_m)=(g_0,\ldots,g_m,f_0,\ldots,f_n)\] for paths $(f_0,\ldots,f_n),(g_0,\ldots,g_m)$ in $\C$ such that $T(g_m)=S(f_0)$.
	
	Finally, we consider the equivalence relation $\sim_\C$ in (the set of arrows of) $F(\C)$ generated by the relations
	\begin{itemize}
		\item $(f,g)=(gf)$ for every pair of arrows $f,g\in\mor(\C)$ such that $t(f)=s(g)$,
		\item $(\id_{c})=()_{c}$ for every object $c\in\obj(\C)$, and
		\item $(f,\overline{f})=()_{s(f)}$ and $(\overline{f},f)=()_{t(f)}$ for every $f\in\mor(\C)$.
	\end{itemize}
	The category $\loc\C$ is defined as the quotient $F(\C)/\sim_\C$ of $F(\C)$ by this equivalence relation. The functor $\iota_{\C}\colon\C\to\loc\C$ is the obvious functor that is the identity on $\obj(\C)=\obj(\loc\C)$ and maps every arrow $f\in\mor(\C)$ to the equivalence class $[(f)]$ of the path $(f)$ in $F(\C)$.

	For any functor $F\colon\C\to \Ds$ between small categories there exists a unique functor $\loc F\colon\loc\C\to\loc\Ds$, which is clearly a groupoid homomorphism, such that $\loc F \iota_\C=\iota_\Ds F$. The assignment $F\mapsto \loc F$ respects identities and compositions, and hence, ``localization'' defines a functor $\loc\colon\cat\to\grpd$, provided a choice of a localization is given (or constructed) for every small category.

	\begin{definition}\label{def:tree}
		Let $T$ be a small category. We say that $T$ is a \emph{tree} if $\loc T$ is an indiscrete category, that is, if there exists exactly one arrow in $\loc T$ from $x$ to $y$, for every $x,y\in \obj(T)$.
	
		A small category $\C$ will be called a \emph{forest} if it is a disjoint union of trees. In other words, $\C$ is a forest if its connected components are trees.
	\end{definition}
	\begin{definition}
		Let $\C$ be a connected small category. A wide subcategory $T$ of $\C$ which is a tree will be called a \emph{maximal tree (of $\C$)}. 
	\end{definition}
	\begin{rem}
		Let $\C$ be a small category. It is not hard to see that there might exist two different maximal trees $T$ and $T'$ in $\C$ with $T$ a proper subcategory of $T'$. In this case, the inclusion $i\colon T\hookrightarrow T'$ induces an isomorphism (of grupoids) $\loc i\colon \loc T\cong \loc T'$.
	\end{rem}
	\begin{lemma}\label{lemm:tree}
		Let $\C$ be a small category, let $T_0$ and $T_1$ be two disjoint trees of $C$ and let $B$ be the disjoint union of $T_0$ and $T_1$. Assume that there exists an arrow $f\colon c_0\to c_1$ in $\C$ such that $c_0\in\obj(T_0)$ and $c_1\in\obj(T_1)$. 
		
		Consider the subcategory $T$ of $\C$ for which $\obj(T)=\obj(T_0)\cup \obj(T_1)$ and for which the set $T(c,c')$ is equal to the set
		\begin{itemize}
			\item ${T_0}(c,c')$ if $c,c'\in\obj(T_0)$,
			\item ${T_1}(c,c')$ if $c,c'\in\obj(T_1)$,
			\item $\{\beta f \alpha:\alpha\in {T_0}(c,c_0),\beta\in {T_1}(c_1,c')\}$ if $c\in\obj(T_0)$ and $c'\in\obj(T_1)$, and
			\item	$\varnothing$ if $c\in\obj(T_1)$ and $c'\in\obj(T_0)$,
		\end{itemize}
		for every $c,c'\in\obj(T)$.
		Then, $T$ is a tree. 
	\end{lemma}
	\begin{proof}
		We wish to show that $\loc T$ is an indiscrete category. 
		We can assume that $\loc T$ is the category $F(T)/\sim_T$ defined at the beginning of the present subsection.

		Since $T$ is connected, $\loc T$ is also connected. Hence it is enough to show that $\aut_{\loc T}(c_0)$ is the trivial group.
		Let $h\in\aut_{\loc T}(c_0)$ and let $h_0,\ldots,h_n$ in $\mor(T)\cup\mor(T)^{-1}$ such that $h=[(h_0,\ldots,h_n)]$. Note that, since any arrow in $T$ from an object $c$ of $T_0$ to an object $c'$ of $T_1$ factors as $\beta f \alpha$ for some $\alpha\in\mor(T_0)$ and some $\beta\in\mor(T_1)$, we can assume that $h_i=f$ for every $i$ such that $S(h_i)\in\obj(T_0)$ and $T(h_i)\in\obj(T_1)$ and that $h_i=\overline{f}$ for every $i$ such that $S(h_i)\in\obj(T_1)$ and $T(h_i)\in\obj(T_0)$.
		Furthermore, there is no loss of generality if we assume that $h_0=f$ and $h_n=\overline{f}$. 
		
		Under this assumptions, we have the followings obvious facts:
		\begin{enumerate}
			\item There exists a unique increasing sequence $0=i_0<\cdots<i_k=n$ such that $h_j\not\in\{f,\overline{f}\}$ whenever $j\not\in\{i_0,\ldots,i_k\}$ and 
			\[
			h_{i_l}=
				\begin{cases}
					f&\text{ if $l$ is even, and}\\
					\overline{f}&\text{ if $l$ is odd.}
				\end{cases}
			\]
			\item The inclusion $\mor(T_0)\subseteq\mor(T)$ induces an obvious identification  of $\mor(T_0)\cup\mor(T_0)^{-1}$ with a subset of $\mor(T)\cup\mor(T)^{-1}$. Under this identification, and since $T_0$ is a tree, every path $(h_i,\ldots,h_{k})$ such that $h_j\neq f$ for $i\leq j\leq k$ and $S(h_i)=T(h_{k})=c_0$ represents the identity $\id_{c_0}$ in $\loc T_0$. And since $\sim_{T_0}\subseteq \sim_{T}$ under the same identification, such a path represents the identity $\id_{c_0}$ in $\loc T$ as well.
			\item Similarly, every path $(h_i,\ldots,h_{k})$ such that $h_j\neq \overline{f}$ for $i\leq j\leq k$ and $S(h_i)=T(h_{k})=c_1$ represents the identity $\id_{c_1}$ in $\loc T$.
		\end{enumerate}
		Now, (3) shows that the path $(h_{i_{2k}},h_{i_{2k}+1},\ldots,h_{i_{2k+1}})$ represents the identity of $c_0$ in $\loc T$ and hence it is equivalent to the path $(h_{i_{2k}},h_{i_{2k+1}})=(f,\overline{f})$ for every $k$. Similarly, (2) shows that the path $(h_{i_{2k+1}},h_{i_{2k+1}+1},\ldots,h_{i_{2k+2}})$ is equivalent to the path $(\overline{f},f)$ for every $k$. By an inductive argument we obtain that the path $(h_0,\ldots,h_n)$ is equivalent to the path $(h_{i_0},\ldots,h_{i_k})=(f,\overline{f},\ldots,f,\overline{f})$ which is in turn equivalent to the empty path in $c_0$. Hence $h$ is the identity $\id_{c_0}$ in $\loc T$. The result follows. 
	\end{proof}
	\begin{prop}\label{prop:forest}
		 Let $\C$ be a connected small category and let $B$ be a subcategory of $\C$ which is a forest. Then $B$ can be extended to a maximal tree $T$ of $\C$.
	\end{prop}
	\begin{proof}
		Without loss of generality, we can assume that $B$ is a wide subcategory of $\C$.
		The set $E$ of (wide) forests of $\C$ that contain $B$ is partially ordered by the relation ``...is a subcategory of...'', or in other words, by the inclusion of subcategories of $\C$. It is straightforward to show that any chain in $E$ has an upper bound in $E$ and thus, it follows from Zorn's lemma that there exists a maximal element $T$ in $E$.
		
		We wish to show that $T$ is connected. Assume the contrary. Since $\C$ is connected and $T$ is wide in $\C$, it is easy to see that there exists and arrow $f$ in $\C$ such that the source and target of $f$ belong to different connected components of $T$. By \ref{lemm:tree}, the forest $T$ can be extended to a forest $T'$ that contains $f$. This fact clearly contradicts the maximality of $T$ in $E$.  
		It follows that $T$ must be connected and thus, that $T$ is a maximal tree of $\C$.
	\end{proof}

	\begin{definition}[{\cite[p.85]{higgins2005categories}}]
		Let $\G$ be a groupoid and let $\n$ be a subgroupoid of $\G$. We say that $\n$ is a \emph{normal} subgroupoid of $\G$ if:
		\begin{enumerate}
			\item $\n$ contains all the identity maps of $\G$ (equivalently, $\n$ is a wide subcategory of $\G$), and
			\item $g^{-1}\aut_{\n}(y)g=\aut_{\n}(x)$ for every $x,y\in \obj(\G)$ and every $g\in\Hom_{\G}(x,y)$.
		\end{enumerate}
		
		Assume that $\n$ is a normal subgroupoid of $\G$. It is clear that $\n$ induces an equivalence relation $\sim$ in $\obj(\G)$ defined as
		\[x\sim y \text{ if and only if $x$ and $y$ are in the same connected component of $\n$.}\]
		The quotient $\obj(\G)/\sim$ is (canonically isomorphic to) the set of connected components of $\n$.
		
		Similarly, $\n$ induces an equivalence relation $\sim$ in $\mor(\G)$ which is defined as
		\[\alpha\sim\beta \text{ if and only if there exist $n_1,n_2\in\mor(\n)$ such that $\alpha=n_1 \beta n_2$}\]
		for $\alpha,\beta\in\mor(\G)$.
		
		We can define a groupoid $\G/\n$ whose set of objects is the quotient $\obj(\G)/\sim$ and whose set of arrows is the quotient $\mor(\G)/\sim$ in the obvious way. The groupoid $\G/\n$ is called the \emph{quotient groupoid of $\G$ by $\n$}. The canonical morphism of groupoids $q\colon \G\to \G/\n$ is called \emph{quotient map of groupoids}.
		
		We will write $[x]$ for the equivalence class of an object $x$ of $\G$ in $\G/\n$. Similarly, we will write $[\alpha]$ for the equivalence class of an arrow $\alpha$ of $\G$ in $\G/\n$.
	\end{definition}
	\begin{prop}[{\cite[Prop 24]{higgins2005categories}}]
		Let $\G$ be a groupoid, let $\n$ be a normal subgroupoid of $\G$ and let $q\colon\G\to\n$ be the canonical quotient map. Let $\mathcal{H}$ be a groupoid and let $F\colon \G\to\mathcal{H}$ be a morphism of groupoids which is trivial in $\n$, that is, $F$ maps every arrow in $\n$ to an identity. Then there exists a unique morphism of groupoids $\overline{F}\colon \G/\n\to \mathcal{H}$ such that $F=\overline{F}q$.
	\end{prop}

	Let $\C$ be a connected small category and let $T$ be a maximal tree in $\C$. It is easy to see that $\loc T$ is a connected normal subgroupoid of $\loc\C$. It follows that $\loc \C/\loc T$ is a one-object groupoid, or in simpler words, a group.
	
	Now, let $c_0\in \obj(\C)$ and let $\rho_\C\colon \loc\C/\loc T\to\aut_{\loc\C}(c_0)$ be the functor defined as $\rho_\C([f])=\beta f\alpha$ for every $f\in \mor(\loc\C)$ where $\alpha$ is the only arrow in $\loc T$ from $c_0$ to $s(f)$ and $\beta$ is the only arrow in $\loc T$ from $t(f)$ to $c_0$.
	A direct computation shows that the quotient map $q_{\C}\colon \loc\C\to \loc\C/\loc T$ is an equivalence of categories whose inverse is the composition
	\[\loc\C/\loc T\xrightarrow{\rho_\C}\aut_{\loc\C}(c_0)\hookrightarrow \loc\C.\]
	In particular, $\rho_\C$ is a group isomorphism whose inverse is the composition
	\[\aut_{\loc\C}(c_0)\hookrightarrow\loc\C\xrightarrow{q_\C}\loc\C/\loc T.\]
	Hence, the quotient $\loc\C/\loc T$ can be canonically identified with the group $\aut_{\loc\C}(c_0)$.

	\section{Regular coverings and fundamental group of Alexandroff spaces}
		\subsection{Fundamental groups of small categories}
		\label{subs:grupo_fundamental}

		With the purpose of fixing notation, we recall the definition of the fundamental groupoid and the fundamental groups of a topological space.
		\begin{definition}\label{def:Pi(X)}
			Let $X$ be a topological space. 
			For every path $\gamma$ in $X$, we write $[\gamma]$ for the path-homotopy equivalence class of $\gamma$. For paths $\gamma_0$ and $\gamma_1$ in $X$ such that $\gamma_0(1)=\gamma_1(0)$, we write $\gamma_0*\gamma_1$ for the concatenation of $\gamma_0$ and $\gamma_1$.
			
			The \emph{fundamental groupoid} of $X$ is the groupoid $\Pi_1(X)$ that has $X$ as its set of objects and the set of path-homotopy equivalence classes of paths from $x$ to $y$ as the set of arrows with source $x$ and target $y$, for every $x,y\in X$. If $[\alpha]\in\Hom_{\Pi_1(X)}(x,y)$ and $[\beta]\in\Hom_{\Pi_1(X)}(y,z)$, the composition $[\beta][\alpha]$ is defined as $[\alpha*\beta]$. 
			
			If $x_0\in X$, we define the \emph{fundamental group of $X$ at $x_0$} as the group \[\pi_1(X,x_0)=\aut_{\Pi_1(X)}(x_0).\]
			Note that this group is the opposite of the usual fundamental group, and hence, it is naturally isomorphic to the latter. 
		\end{definition}
		
		D. Quillen proved in \cite{quillen1973higher} that there exists a canonical isomorphism between the fundamental group $\pi_1(B\C,c_0)$ of a small category $\C$ at an object $c_0$ and the group $\aut_{\loc\C}(c_0)$. Next, we will give an explicit description of this isomorphism.
		
		\begin{definition}\label{def:zeta}
			Let $\C$ be a small category and let $c_0\in\obj(\C)$. We will define a group homomorphism $\zeta_{(\C,c_0)}$ from the group $\aut_{\loc\C}(c_0)$ to the fundamental group $\pi_1(B\C,c_0)$ of $B\C$ at $c_0$.
			
			First, note that if we identify the topological $1$-simplex $\Delta^{1}$ with the unit interval $I$ in the obvious way, then any arrow $f\colon c\to c'$ in $\C$ induces a canonical path $\gamma_f\colon I\to B\C$ defined as $\gamma_{f}(t)=[(\sigma_{f},t)]$ for every $t\in I$, where $\sigma_f$ is the $1$-simplex of $N\C$ that corresponds to $f$. It is clear that $\gamma_{f}$ is a path in $B\C$ from (the $0$-cell that corresponds to) $c$ to (the $0$-cell that corresponds to) $c'$.
			
			There is an obvious covariant functor $\xi_\C\colon\C\to\Pi_1(B\C)$ that maps every $f\in\mor(\C)$ to the path-homotopy class $[\gamma_f]$ of $\gamma_f$. Since $\Pi_1(B\C)$ is a groupoid, this functor induces a functor $\overline{\xi}_\C\colon\loc \C\to\Pi_1(B\C)$ that restricts to a functor (and thus, a group homomorphism) 
			\[\zeta_{(\C,c_0)}\colon\aut_{\loc\C}(c_0)\to\pi_1(B\C,c_0).\]
		\end{definition}
		\begin{prop}\label{prop:quillen_aut_pi}		
			Let $\C$ be a small category and let $c_0\in\obj(\C)$. Let $\zeta=\zeta_{(\C,c_0)}\colon \aut_{\loc\C}(c_0)\to\pi_1(B\C,c_0)$ be the group homomorphism defined in \ref{def:zeta}. Then $\zeta$ is the canonical isomorphism from \cite[Proposition 1]{quillen1973higher}. 
		\end{prop}
		\begin{proof}
			We will follow Quillen's proof of \cite[Proposition 1]{quillen1973higher} and give explicit constructions for the involved functors.
			
			Without loss of generality we can assume that $\C$ is connected. We will write
			\begin{itemize}
				\item $G$ for the group $\aut_{\loc\C}(c_0)$,
				\item $\Ds$ for the full subcategory of $\set^{\C}$ whose objects are the morphism-inverting functors from $\C$ to $\set$,
				\item $\cov(B\C)$ for the category of coverings of $B\C$, and
				\item $G\text{-}\set$ for the category of $G$-sets.
			\end{itemize}

			Quillen constructs equivalences of categories
			\[\cov(B\C)\simeq \Ds\cong\set^{\loc \C}\simeq\set^G\cong G\text{-}\set\]
			that we will try to describe.
			
			The equivalence $\cov(B\C)\simeq \Ds$ is constructed as follows.
			Given a covering map $p\colon E\to B\C$ and an arrow $f\colon c\to c'$ in $\C$, there is a bijection $E(f)\colon p^{-1}(c)\to p^{-1}(c')$ induced by the action of $[\gamma_f]$ that maps every element $\tilde{c}$ of $p^{-1}(c)$ to the end point of the unique lift of $\gamma_f$ by $p$ starting at $\tilde{c}$. The assignment $f\mapsto E(f)$ defines a morphism-inverting functor $M_p\colon\C\to\set$ induced by $p$.
			
			The equivalence $\cov(B\C)\simeq \Ds$ is the functor that maps every covering $p$ of $B\C$ to the associated morphism-inverting functor $M_p$. An inverse to this equivalence is constructed by mapping every morphism-inverting functor $F\colon \C\to\set$ to the covering projection
			\[B(u_F)\colon B(*\downarrow F)\to B\C\]
			induced by the projection $u_F$ from the comma category $*\downarrow F$ to $\C$.

			The isomorphism $\Ds\cong\set^{\loc \C}$ is the functor that maps every morphism-inverting functor $M\colon \C\to \set$ to the unique functor $\overline{M}\colon \loc \C\to\set$ such that $\overline{M}\iota_{\C}=M$.
			
			The equivalence $\set^{\loc\C}\simeq\set^{G}$ is induced by the inclusion $G\hookrightarrow \loc\C$ which is an equivalence of categories since $\C$ is connected.
			
			The isomorphism $\set^{G}\cong G$-$\set$ is the obvious isomorphism that maps a functor $F\colon G\to\set$ to the corresponding $G$-set (which is nothing but the set $F(c_0)$ equipped with the left $G$-action induced by $F$). 
			
			Now, consider the group $G$ as a $G$-set acting on itself by left translation and let $p\colon \widetilde{B\C}\to B\C$ be the covering of $B\C$ that corresponds to $G$ through the equivalence $G$-$\set\simeq\cov(B\C)$. Note that $p^{-1}(c_0)$ is the discrete space $\{c_0\}\times G$. Let $\tilde{c}_0=(c_0,\id_{c_0})\in p^{-1}(c_0)$.
			
			Given a path $\gamma$ in $B\C$ and $x\in p^{-1}(\gamma(0))$, we will write $\widetilde{\gamma}^{x}$ for the unique lift of $\gamma$ by $p$ from $x$. The canonical identification $\{c_0\}\times G\cong G$ shows that 
			\[(c_0,gh)=g\cdot (c_0,h)=\widetilde{\gamma_g}^{(c_0,h)}(1)\] for some $\gamma_g\in\zeta(g)$, for every $g,h\in G$. 
			
			Since the action of $G$ is transitive, it is clear that $\widetilde{B\C}$ is a connected covering of $B\C$. On the other hand, any non-empty $G$-set $X$ defines a $G$-map $f\colon G\to X$ by $f(g)=g\cdot x$ for some fixed $x\in X$. It is clear then that $p$ factors through any covering of $B\C$, from where it follows that $p$ is the universal cover of $B\C$.
		
			Now, the equivalence $\cov(B\C)\simeq G\text{-}\set$ restricts to an isomorphism 
			\[\aut_{\cov(B\C)}(p)\cong \aut_{G\text{-}\set}(G).\]
			Note that $\aut_{\cov(B\C)}(p)$ is nothing but the deck transformation group $\deck(p)$ of $p$. From the proof of  \cite[Proposition 1.39]{hatcher2002algebraic}, it follows that there exist an isomorphism  
			\[\pi_1(B\C,c_0)^{\op}\cong \deck(p)\] that maps the element $[\gamma]$ to the unique deck transformation $\tau_\gamma$ of $p$ such that 
			\[\tau_\gamma\big(\tilde{c}_0\big)=\widetilde{\gamma}^{\tilde{c}_0}(1).\]
			
			On the other hand, the automorphism group of $G$ in $G$--$\set$ is canonically isomorphic to $G^{\op}$ through the isomorphism that maps every $g\in G$ to the bijection $\sigma_g\colon G\to G$ defined by 
			\[\sigma_g(h)=hg\] for every $h\in G$. 
			The isomorphism
			\[G^{\op}\cong \aut_{G\text{-}\set}(G)\cong \deck(p)\cong\pi_1(B\C,c_0)^{\op}\] is the opposite of the isomorphism from Quillen's proof.
			
			Now, for every $g\in G$, the bijection $\sigma_g$ maps the identity $\id_{c_0}$ to $g$, from where it follows that the deck transformation of $p$ that corresponds to $g$ is the unique deck transformation $\tau^{g}$ of $p$ that maps  $\tilde{c}_0=(c_0,\id_{c_0})$ to $(c_0,g)=\widetilde{\gamma_g}^{\tilde{c}_0}(1)$. 
			Hence, every $g\in G$ maps as 
			\[g\mapsto \tau^{g}=\tau_{\gamma_g}\mapsto [\gamma_g]=\zeta(g)\] through the composition 
			$G\cong\deck(p)^{\op}\cong\pi_1(B\C,c_0)$. Thus, the isomorphism $G\cong\pi_1(B\C,c_0)$ is precisely the group homomorphism $\zeta$, as we wished to show.
		\end{proof}
		\begin{prop}\label{prop:naturality}
			The functors $\xi_\C$ and $\overline{\xi}_\C$ from Definition \ref{def:zeta} are natural in $\C$ for $\C\in\cat$. The functor $\zeta_{(\C,c_0)}$ is natural in $(\C,c_0)$ for $(\C,c_0)\in\cat_*$.
		\end{prop}
		\begin{proof}
			Let $\C$ and $\Ds$ be small categories and let $F\colon\C\to\Ds$ be a functor. Let $f$ be an arrow in $\C$.
			Note that $\sigma_{F(f)}=F\sigma_f$. 
			It follows that 
			\[\gamma_{F(f)}(t)=[\sigma_{F(f)},t]=[F\sigma_f,t]=BF(\gamma_f(t))=(BF\gamma_f)(t)\]
			for every $t\in I$. Hence, $\gamma_{F(f)}=BF\gamma_f$. Thus,
			\[\Pi_1(BF)\xi_{\C}(f)=\Pi_1(BF)([\gamma_f])=[BF\gamma_f]=[\gamma_{F(f)}]=\xi_{\Ds}(F(f))\]
			Therefore, $\xi$ is a natural transformation from the inclusion functor $\ATop\hookrightarrow\cat$ to the composition
			\[\ATop\hookrightarrow\cat\xrightarrow{B}\Top\xrightarrow{\Pi_1}\grpd\hookrightarrow{\cat}.\]
			
			The naturality of the functor $\overline{\xi}$ follows easily by a straightforward ``arrow trick''. Indeed, 
	\[\Pi_1(BF)\overline{\xi}_{\C}\iota_{\C}=\Pi_1(BF)\xi_{\C}=\xi_{\Ds}F=\overline{\xi}_{\Ds}\iota_{\Ds}F=\overline{\xi}_{\Ds}\loc F\iota_{\C}\]
			and hence $\Pi_1(BF)\overline{\xi}_{\C}=\overline{\xi}_{\Ds}\loc F$. 
			
			The naturality of the functor $\zeta$ is obvious.
		\end{proof}
		
		\begin{coro}\label{coro:zeta_nat}
			The natural transformation  
			\[\zeta\colon \aut_{\loc(\cdot)}(\cdot)\cong \pi_1(B(\cdot),\cdot)\colon \cat_*\to\grp.\]
			is a natural isomorphism.
		\end{coro}
		\begin{proof}
			Immediate from \ref{prop:quillen_aut_pi} and \ref{prop:naturality}.
		\end{proof}
		\begin{ex}\label{exam:K(G,1)}
			Let $G$ be a group, and consider $G$ as a category with a unique object $*$. 
			Let $\varphi\in\set^{G}$ be the functor associated to $G$ when considering $G$ as a left $G$-set in the usual way, that is, equipped with the left action of $G$ on itself by left translation. 
			
			From \ref{prop:quillen_aut_pi} it follows that $\zeta_{(G,*)}\colon G=\loc G\to\pi_1(BG,*)$ is an isomorphism. 
			
			The comma category $*\downarrow \varphi$ can be computed explicitly. Its objects are in canonical bijection with the elements of $G$, and there exists a unique arrow (corresponding to the element $hg^{-1}$) between the objects that correspond to  $g$ and $h$, for every $g,h\in G$. Thus, $*\downarrow \varphi$ is an indiscrete category and hence $B(*\downarrow \varphi)$ is contractible. Since $B(u_\varphi)\colon B(*\downarrow \varphi)\to BG$ is a covering map, it is clear that $BG$ is aspherical. Therefore, $BG$ is an Eilenberg-MacLane space of type $(G,1)$.
		\end{ex}
		\begin{rem}\label{rema:iso_nat_quillen}
			From \ref{coro:zeta_nat} and \ref{exam:K(G,1)} it follows that the functor \[\grp\hookrightarrow\cat_*\xrightarrow{B}\Top_*\xrightarrow{\pi_1}\grp\] is naturally isomorphic to the identity of $\grp$.
	\end{rem}
	If $X$ is an Alexandroff space, the space $BX$ is weakly equivalent to $X$ and hence the fundamental groupoid of $BX$ is equivalent to the fundamental groupoid of $X$. Furthermore, it is well known that the fundamental groupoid of $BX$ is equivalent to $\loc X$ (\cite[Chapter III, Section 1, Corollary 1.2]{goerss2009simplicial}). It follows that the groupoid $\Pi_1(X)$ is equivalent to $\loc X$. A stronger result holds: the groupoids $\Pi(X)$ and $\loc X$ are naturally isomorphic. We give a proof of this fact that does not rely on deep results of simplicial homotopy theory, and thus, can be considered elementary.
	\begin{definition}\label{def:eta}
		Let $X$ be an Alexandroff space and let $x,y\in X$ such that $x\leq y$. We define the path $\eta(x\leq y)\colon I\to X$ as 		\begin{displaymath}
			\eta(x\leq y)(t)=\begin{cases} x&\text{ if $t<1$,}\\ y&\text{ if $t=1$.}\end{cases}
		\end{displaymath}
	\end{definition}
	\begin{theo}\label{theo:loc_iso_pi}
		The functors $\loc\colon \ATop\to \grpd$ and $\Pi_1\colon\ATop\to\grpd$ are naturally isomorphic. In particular, for every pointed Alexandroff space $(X,x_0)$ there exists a natural isomorphism $\aut_{\loc X}(x_0)\cong \pi_1(X,x_0)$.
		\end{theo}
		\begin{proof}
			For every Alexandroff space $X$, we consider the functor $Z_X\colon X\to \Pi_1(X)$ that is the identity on objects and that maps every arrow $x\leq x'$ in $X$ to the path-homotopy class $[\eta(x\leq x')]$.
			
			If $X$ and $Y$ are Alexandroff spaces and $f\colon X\to Y$ is a continuous function, then $f\circ \eta(x\leq x')=\eta(f(x)\leq f(x'))$ for every $x,x'\in X$ such that $x\leq x'$. It follows that the assignment $X\mapsto Z_X$ defines a natural transformation $Z$ from the functor \[\ATop\cong\ord\hookrightarrow \cat\] to the functor \[\ATop\hookrightarrow \Top\xrightarrow{\Pi_1}\grpd\hookrightarrow\cat.\]			
			
			The factorization of $Z_X$ through $\iota_X$ induces a functor $\overline{Z}_X\colon \loc X\to \Pi_1(X)$ for every Alexandroff space $X$. It is  not hard to see that the assignment $X\mapsto \overline{Z}_X$ defines a natural transformation from the functor 
			\[\ATop\cong\ord\hookrightarrow \cat\xrightarrow{\loc}\grpd\] to the functor \[\ATop\hookrightarrow \Top\xrightarrow{\Pi_1}\grpd.\]
			
			We will prove that $\overline{Z}_X$ is an isomorphism for every Alexandroff space $X$.
			
			Let $X$ be an Alexandroff space. By standard arguments, and since $\overline{Z}_X$ is the identity on objects, it is enough to show that $\overline{Z}_X$ restricts to isomorphisms $\overline{Z}_X|\colon\aut_{\loc X}(x_0)\cong\pi_1(X,x_0)$ for every $x_0\in X$.
			
			Let $x_0\in X$ and let $\varphi_X\colon BX\to X$ be the weak equivalence defined in Definition \ref{def:varphi_X}. For every arrow $x\leq x'$ in $X$, we write $\gamma_{x\leq x'}$ for the canonical path in $BX$ from $x$ to $x'$ defined in \ref{def:zeta}.
			
			A direct computation shows that $\eta(x\leq x')=\varphi_X\circ \gamma_{x\leq x'}$ for every arrow $x\leq x'$ in $X$. It follows that
			\[Z_X(x\leq x')=[\eta(x\leq x')]=[\varphi_X\circ \gamma_{x\leq x'}]=\Pi_1(\varphi_X)([\gamma_{x\leq x'}])=\Pi_1(\varphi_X)\overline{\xi}_X\iota_X(x\leq x')\] for every $x,x'\in X$ such that $x\leq x'$, where $\overline{\xi}_X$ is the functor from Definition \ref{def:zeta}.
			
			It is immediate that $\overline{Z}_X=\Pi_1(\varphi_X)\overline{\xi}_X$. 
			Hence, the diagram from Figure \ref{fig:diagram} commutes.
			
			\begin{figure}[!ht]
				\begin{center}
					\begin{tikzpicture}[x=3cm,y=2cm]
						\draw (0,0) node(LX){$\loc X$};
						\draw (1,0) node(PiBX){$\Pi_1(BX)$};
						\draw (2,0) node(PiX){$\Pi_1(X)$};
						\draw (0,1) node(autLX){$\aut_{\loc X}(x_0)$};
						\draw (1,1) node(piBX){$\pi_1(BX,x_0)$};
						\draw (2,1) node(piX){$\pi_1(X,x_0)$};
						\draw (1,-0.9) node{$\overline{Z}_X$};
						\draw[->] (LX) -- (PiBX) node [midway,below] {$\overline{\xi}_X$};
						\draw[->] (PiBX) -- (PiX) node [midway,below] {$\Pi_1(\varphi_X)$};
						\draw[->] (autLX) -- (piBX) node [midway,above] {$\zeta_{(X,x_0)}$};
						\draw[->] (piBX) -- (piX) node [midway,above] {$\pi_1(\varphi_X)$};
						\draw[->,bend right=45] (LX) to (PiX);
						\draw[right hook->] (autLX) -- (LX);
						\draw[right hook->] (piBX) -- (PiBX);
						\draw[right hook->] (piX) -- (PiX);
					\end{tikzpicture}
				\end{center}
				\caption{}\label{fig:diagram}
			\end{figure}
			
			It follows that the restriction $\overline{Z}_X|\colon \aut_{\loc X}(x_0)\to \pi_1(X,x_0)$ of $\overline{Z}_X$ is nothing but the composition	$\pi_1(\varphi_X)\zeta_{(X,x_0)}$ which is clearly an isomorphism.
			
			The result follows.
		\end{proof}

		\subsection{Regular coverings of Alexandroff spaces}
		\label{subs:coverings}

		For every connected small category $\C$, every object $c_0\in\obj(\C)$ and every maximal tree $T$ of $\C$, we have a functor	$\mathcal{F}_{T}\colon\C\to \pi_1(\C,c_0)$ induced by $T$ which is the composition 
		\[\C\xrightarrow{\iota_\C}\loc \C\xrightarrow{q_\C} \loc \C/\loc T\xrightarrow{\cong} \pi_1(\C,c_0)\]
		(although $\mathcal{F}_{T}$ depends on $c_0$ as well as on $T$, we decided to leave the object $c_0$ out of the notation for the sake of simplicity).
		
		Given a group homomorphism $\alpha\colon\pi_1(\C,c_0)\to G$, the functor $\alpha\circ\mathcal{F}_T$ will be denoted by $\mathcal{F}_{T,\alpha}$.
		
		Note that the functor $\mathcal{F}_{T,\alpha}$ induces a group homomorphism
		\[\pi_1(\C,c_0)\xrightarrow{(\mathcal{F}_{T,\alpha})_*} \pi_1(G,*)\cong G,\]
		where $\pi_1(G,*)\cong G$ is the canonical isomorphism from Example \ref{exam:K(G,1)}
		
		\begin{lemma}\label{lemm:pi=aut}
			With the above definitions, the composition
			\[\pi_1(\C,c_0)\xrightarrow{(\mathcal{F}_{T,\alpha})_*} \pi_1(G,*)\cong G\] is the group homomorphism $\alpha$.
		\end{lemma}
		\begin{proof}
			We will write $\xrightarrow{\cong}$ for the canonical isomorphisms described so far. This notation should not be a cause of confusion. 
			
			It is easy to see that $\iota_\C\colon \C\to\loc\C$ induces the identity 
			\[\aut_{\loc\C}(c_0)\to\aut_{\loc\C}(c_0)\]
			and thus, it induces the isomorphism
			\[\pi_1(\C,c_0)\xrightarrow{\cong}\aut_{\loc\C}(c_0)\xrightarrow{\cong}\pi_1(\loc\C,c_0).\]
			
			Similarly, the inclusion functor $\aut_{\loc\C}(c_0)\hookrightarrow\loc\C$ induces the identity morphism
			\[\aut_{\loc\C}(c_0)\to\aut_{\loc\C}(c_0)\]
			as well, and thus, it induces the group homomorfism
			\[\pi_1(\aut_{\loc\C}(c_0),c_0)\xrightarrow{\cong}\aut_{\loc\C}(c_0)\xrightarrow{\cong}\pi_1(\loc\C,c_0).\]
			
			As we mentioned earlier, the composition
			\[\aut_{\loc\C}(c_0)\hookrightarrow\loc\C\xrightarrow{q_\C}\loc\C/\loc T\xrightarrow{\rho_\C}\aut_{\loc\C}(c_0)\] is the identity. It follows that the functor $\rho_\C q_\C\colon \loc\C\to\aut_{\loc\C}(c_0)$ induces the isomorphism
			\[\pi_1(\loc\C,c_0)\xrightarrow{\cong}\aut_{\loc\C}(c_0)\xrightarrow{\cong}\pi_1(\aut_{\loc\C}(c_0),c_0)\]
			between the fundamental groups.
			
			From \ref{rema:iso_nat_quillen}, we see that the canonical isomorphism $\aut_{\loc\C}(c_0)\to\pi_1(\C,c_0)$ induces the isomorphism
			\[\pi_1(\aut_{\loc\C}(c_0),c_0)\xrightarrow{\cong}\aut_{\loc\C}(c_0)\xrightarrow{\cong}\pi_1(\C,c_0)\xrightarrow{\cong}\pi_1(\pi_1(\C,c_0),*)\] between the fundamental groups.
			
			By composing the above induced morphisms, we obtain that the morphism
			\[\C\xrightarrow{\iota_\C}\loc\C\xrightarrow{\rho_\C q_\C}\aut_{\loc\C}(c_0)\xrightarrow{\cong}\pi_1(\C,c_0)\] induces the isomorphism
			\[\pi_1(\C,c_0)\xrightarrow{\cong}\pi_1(\pi_1(\C,c_0),*)\] between the fundamental groups.
			
			By naturality of $\zeta$ we obtain that the morphism induced by $\mathcal{F}_{T,\alpha}\colon\C\to G$ between the fundamental groups is the morphism
			\[\pi_1(\C,c_0)\xrightarrow{\alpha} G\xrightarrow{\cong}\pi_1(G,*).\]
			The result follows by composing the above isomorphism with the isomorphism $\pi_1(G,*)\xrightarrow{\cong} G$.
		\end{proof}
		In particular, $(\mathcal{F}_{T})_*=\id_{\pi_1(\C,c_0)}$.
		
		The following corollary is easy to obtain.
		\begin{coro}\label{coro:F_cte_sii_alpha_0}
			With the above definitions, the functor $\mathcal{F}_{T,\alpha}$ is trivial if and only if $\alpha=0$.
		\end{coro}
		\begin{proof}
			Since $\mathcal{F}_{T,\alpha}$ factors through $\alpha$, it will be trivial if $\alpha=0$. On the other hand, if $\mathcal{F}_{T,\alpha}$ is trivial, then $B\mathcal{F}_{T,\alpha}$ is a constant map. From the previous lemma, it follows that  $\alpha=(\mathcal{F}_{T,\alpha})_*=0$.
		\end{proof}
		
		Let $\C$ and $\Ds$ be two connected small categories, let $T_\C$ and $T_\Ds$ be maximal trees in $\C$ and $\Ds$, respectively, and let $F\colon\C\to\Ds$ be a functor such that $F(T_\C)\subseteq T_\Ds$. Given $c_0\in\obj(\C)$, we have the following commutative diagram:
		\begin{figure}[!ht]
			\begin{center}
				\begin{tikzpicture}[x=3cm,y=2.5cm]
					\draw (0,1) node(C){$\C$};
					\draw (1,1) node(LC){$\loc \C$};
					\draw (2,1) node(LCLT){$\loc \C/\loc T_\C$};
					\draw (3,1) node(pi1C){$\pi_1(\C,c_0)$};
					\draw (0,0) node(D){$\Ds$};
					\draw (1,0) node(LD){$\loc \Ds$};
					\draw (2,0) node(LDLT){$\loc \Ds/\loc T_\Ds$};
					\draw (3,0) node(pi1D){$\pi_1(\Ds,F(c_0))$};
					\draw[->] (C) -- (LC) node [midway,above] {$\iota_\C$};
					\draw[->] (LC) -- (LCLT) node [midway,above] {$q_\C$};
					\draw[->] (LCLT) -- (pi1C) node [midway,above] {$\cong$};
					\draw[->] (D) -- (LD) node [midway,below] {$\iota_\Ds$};
					\draw[->] (LD) -- (LDLT) node [midway,below] {$q_\Ds$};
					\draw[->] (LDLT) -- (pi1D) node [midway,below] {$\cong$};
					\draw[->] (C) -- (D) node [midway,left] {$F$};
					\draw[->] (LC) -- (LD) node [midway,left]{$\loc F$};
					\draw[->] (LCLT) -- (LDLT) node [midway,right]{$\overline{\loc F}$};
					\draw[->] (pi1C)-- (pi1D) node [midway,right]{$F_*$};
				\end{tikzpicture}
			\end{center}
			\label{figu:diagrama}
		\end{figure}
	
		Suppose that $X$ is a connected Alexandroff space, that $A$ is a connected subspace of $X$, that there exists a maximal tree $T$ of $X$ whose intersection with $A$ is a maximal tree $T_A$ of $A$ and that we have a group homomorphism  $\alpha\colon \pi_1(X,a_0)\to G$ for some $a_0\in A$ and some group $G$.
		
		Let $i\colon A\to X$ be the inclusion and let $i_*\colon\pi_1(A,a_0)\to \pi_1(X,a_0)$ be the group homomorphism induced by $i$. From the above diagram, we see that $\mathcal{F}_{T_A,\alpha i_*}$ is the restriction of $\mathcal{F}_{T,\alpha}$ to $A$. From Corollary \ref{coro:F_cte_sii_alpha_0} it follows that $\mathcal{F}_{T,\alpha}$ is trivial on $A$ if and only if $i_\ast(\pi_1(A,a_0))\subseteq \ker \alpha$.
		
		Now, if $A$ is not connected, we will have to consider several basepoints, at least one for each connected component of $A$, to determine whether the functor $\mathcal{F}_{T,\alpha}$ is trivial on $A$. Hence, we will have to consider several basepoints on $X$.
	
		\begin{prop}\label{prop:i_ker_F_trivial}
			Let $X$ be a connected Alexandroff space, let $A$ be a non-empty subspace of $X$ and let $i\colon A\to X$ be the inclusion map. Let $T$ be a maximal tree on $X$, let $x_0\in X$ and let $\alpha\colon \pi_1(X,x_0)\to G$ be a group homomorphism for some group $G$.
			Suppose that the intersection of $T$ with each connected component of $A$ is a tree, and hence, a maximal tree on said component.
			
			For each $a\in A$, let $\theta_a\colon\pi_1(X,a)\cong\pi_1(X,x_0)$ be the isomorphism induced by the path-homotopy class of some path in $X$ from $a$ to $x_0$ that corresponds to the unique arrow in $\loc T$ from $a$ to $x_0$.
			
			Then, the functor $\mathcal{F}_{T,\alpha}$ is trivial on $A$ if and only if $i_*(\pi_1(A,a))\subseteq\ker\alpha \theta_a$ for every $a\in A$.
		\end{prop}
		\begin{proof}
			Let $a\in A$. 
			Consider the following diagram in the category $\cat_*$
				\begin{center} 
					\begin{tikzpicture}[x=2.7cm,y=2.5cm]
						\draw (1,1) node(LXLTa){$(\loc X/\loc T,[a])$};
						\draw (1,0) node(LXLTx){$(\loc X/\loc T,[x_0])$};
						\draw (2.5,1) node(auta){$\aut_{\loc X}(a)$};
						\draw (4,1) node(piXa){$\pi_1(X,a)$};
						\draw (5,1) node(G){$G$};
						\draw (2.5,0) node(autx){$\aut_{\loc X}(x_0)$};
						\draw (4,0) node(piXx){$\pi_1(X,x_0)$};
						\draw[double equal sign distance] (LXLTa)--(LXLTx) node [midway,left] {};
						\draw[->] (LXLTx)--(autx) node [midway,below] {$\cong$};
						\draw[->] (auta)--(piXa) node [midway,above] {$\cong$};
						\draw[->] (piXa)--(G) node [midway,above] {$\alpha \theta_a$};
						\draw[->] (LXLTa)--(auta) node [midway,above] {$\cong$};
						\draw[->] (autx)--(piXx) node [midway,below] {$\cong$};
						\draw[->] (piXx)--(G) node [midway,below right] {$\alpha$};
						\draw[->] (auta)--(autx) node [midway,left] {$\cong$};
						\draw[->] (piXa)--(piXx) node [midway,left] {$\theta_a$};
					\end{tikzpicture}
				\end{center}
				where the isomorphism $\aut_{\loc X}(a)\cong\aut_{\loc X}(x_0)$ is the obvious isomorphism induced by the only arrow in $\loc T$ from $a$ to $x_0$. It is easy to see that the diagram commutes, from where it follows that $\mathcal{F}_{T,\alpha}=\mathcal{F}_{T,\alpha \theta_a}$.
				
				Write $T_a$ for the tree that is obtained by intersecting $T$ with the connected component of $A$ that contains $a$.
				Note that the restriction $\mathcal{F}_{T,\alpha \theta_a}$ to $A$ is nothing but the functor $\mathcal{F}_{T_a,\alpha \theta_a i_*}$
				Now, $\mathcal{F}_{T,\alpha}$ is trivial on the connected component of $A$ that contains $a$ if and only if $\mathcal{F}_{T,\alpha \theta_a}$ is trivial on said component, and hence, if and only if the functor $\mathcal{F}_{T_a,\alpha \theta_a i_*}$ is trivial. By Corollary \ref{coro:F_cte_sii_alpha_0}, this happens if and only if $\alpha \theta_a i_*=0$, or equivalently, if and only if $i_*(\pi_1(A,a))\subseteq \ker \alpha \theta_a$. The result easily follows.
		\end{proof}
		\begin{coro}
			Let $X$ be a connected Alexandroff space and let $A$ be a subspace of $X$ such that the inclusion map $i\colon A\to X$ induces the trivial map $i_*\colon \pi_1(A,a)\to\pi_1(X,a)$ for every basepoint $a\in A$. Then there exists a maximal tree $T$ in $X$ such that $\mathcal{F}_{T}$ is trivial on $A$. 
		\end{coro}
		\begin{proof}
			We define a forest $B$ in $X$ by choosing a maximal tree in each connected component of $A$. By \ref{prop:forest}, we can extend $B$ to a maximal tree $T$ of $X$. It is not hard to see that the hypotheses of Proposition \ref{prop:i_ker_F_trivial} are satisfied from where it follows that $\mathcal{F}_{T}$ is trivial on $A$.
		\end{proof}
		The next theorem and its corollary show that the functors $\mathcal{F}_{T,\alpha}$ can be used to represent the regular coverings of $X$ (compare with \cite[Theorem 3.6]{barmak2012colorings}). The importance of said functors being trivial on certain subspaces of $X$ will be made explicit. If $X$ is an Alexandroff space, $T$ is a maximal tree of $X$, $\alpha\colon \pi_1(X,x_0)\to G$ is a homomorphism of groups for some $x_0\in X$ and some group $G$ and $\mathcal{F}_{T,\alpha}$ is trivial on a subspace $A$ of $X$, then the inverse image of $A$ through the corresponding covering map is nothing but the product $A\times G$ (where $G$ is considered as a discrete topological space) and thus, is a disjoint union of copies of $A$ indexed by $G$. This fact was already used in \cite{cianci2018splitting} and was a key step to our results about the minimal finite models of the real projective plane, the torus and the Klein bottle.
		
		\begin{theo}\label{theo:rev_reg_de_alex}
			Let $X$ be a connected Alexandroff space and let $x_0\in X$. 
			Let $G$ be a group and consider it as a category with a unique object $*$. Let $\mathcal{F}\colon X\to G$ be a functor.
			Let $\widetilde{X}=\mathcal{F}\downarrow *$ (where $*$ denotes the only possible functor from the terminal category to $G$), let $p\colon\widetilde{X}\to X$ be the canonical projection and let $\tilde{x}_0\in p^{-1}(x_0)$.
			Let $\mathcal{F}_*\colon\pi_1(X,x_0)\to G$ be the map induced by $\mathcal{F}$ between the fundamental groups.
			Then, the projection $p$ is the regular covering map of $X$ that corresponds to the kernel of the map $\mathcal{F}_*$, with $\pi_0(\widetilde{X},\tilde{x}_0)\cong\coker\mathcal{F}_*$. 
		
			If, in particular, $\mathcal{F}=\mathcal{F}_{T,\alpha}$ for some maximal tree $T$ of $X$ and some group homomorphism $\alpha\colon \pi_1(X,x_0)\to G$, then $p$ is the regular covering map of $X$ that corresponds to the group $\ker\alpha$, with $\pi_0(\widetilde{X},\tilde{x}_0)\cong\coker\alpha$. 
		\end{theo}
		\begin{proof}
			A direct computation shows that $\obj(\widetilde{X})=X\times G$ and that for every $(x,g),(x',g')\in\widetilde{X}$ there exists a unique arrow in $\widetilde{X}$ from $(x,g)$ to $(x',g')$ if and only if $x\leq x'$ in $X$ and $g'\mathcal{F}(x\leq x')=g$ in $G$. It follows that $\widetilde{X}$ is an Alexandroff space whose preorder is defined by
			\[(x,g)\leq (x',g')\Leftrightarrow x\leq x'\text{ and } g=g'\mathcal{F}(x\leq x')\]
			for $(x,g),(x',g')\in \widetilde{X}$ and that $p$ is a continuous map.
			
			We will show that for every $x\in X$
			\begin{enumerate}
				\item the set $p^{-1}(U_x)$ is the union $\bigcup\limits_{g\in G} U_{(x,g)}$,
				\item the sets $U_{(x,g)}$ with $g\in G$ are mutually disjoint, and
				\item the set $U_{(x,g)}$ is mapped homeomorphically on $U_x$ by $p$ for every $g\in G$.
			\end{enumerate}

			Let $x\in X$. We will show (1). 
			If $(x',g')\in p^{-1}(U_x)$ then $x'\leq x$ and therefore $(x',g')\leq (x,g)$ where $g=g'\mathcal{F}(x'\leq x)^{-1}$. It follows that $p^{-1}(U_x)\subseteq\bigcup\limits_{g\in G} U_{(x,g)}$. The other inclusion is obvious.
			
			It is easy to see (2). Indeed, if $(a,h)\in U_{(x,g)}\cap U_{(x,g')}$ for some $g,g'\in G$, then 
			\[g\mathcal{F}(a\leq x)=h=g'\mathcal{F}(a\leq x)\] from where it follows that $g=g'$. 
			
			We will now prove (3). Let $g\in G$. The function $\phi_g\colon U_x\to U_{(x,g)}$ defined by 
			\[\phi_g(x')=(x',g\mathcal{F}(x'\leq x))\] for $x'\in U_x$ is an inverse of the restriction $p|\colon U_{(x,g)}\to U_x$ of $p$.
			
			We conclude that $p$ is a covering map of $X$.
			
			Now, for every $g\in G$, the functor $g_*\colon \widetilde{X}\to\widetilde{X}$ of \cite[Theorem B]{quillen1973higher} induced by $g$ is nothing but the deck transformation of $p$ defined by
			\[g_*(x,h)=(x,gh)\] for every $(x,h)\in\widetilde{X}$. It follows that $B(g_*)$ is a homeomorphism for every $g\in G$. By a direct application of Quillen's Theorem $B$ we obtain the following part of a long exact sequence
			\[\pi_2(G,*)\to\pi_1(\widetilde{X},(x_0,e))\xrightarrow{p_*}\pi_1(X,x_0)\xrightarrow{\mathcal{F}_*}\pi_1(G,*)\to\pi_0(\widetilde{X},(x_0,e))\xrightarrow{p_*}\pi_0(X,x_0)\] where $e$ is the identity of $G$.
			By \ref{exam:K(G,1)} we see that there exists an exact sequence 
			\[0\to\pi_1(\widetilde{X},(x_0,e))\xrightarrow{p_*}\pi_1(X,x_0)\xrightarrow{\mathcal{F}_*}G\to\pi_0(\widetilde{X},(x_0,e))\to 0.\]
			It follows that $p_*(\pi_1(\widetilde{X},(x_0,e)))=\ker\mathcal{F}_*$ and that $\pi_0(\widetilde{X},(x_0,e))\cong\coker\mathcal{F}_*$.
			
			It is clear that $p_*(\pi_1(\widetilde{X},\tilde{x}_0))$ is a conjugate of $p_*(\pi_1(\widetilde{X},(x_0,e)))=\ker\mathcal{F}_*$ and hence, that $p_*(\pi_1(\widetilde{X},\tilde{x}_0))=\ker\mathcal{F}_*$ as we wished to show. On the other hand, $\pi_0(\widetilde{X},\tilde{x}_0)\cong\pi_0(\widetilde{X},(x_0,e))\cong\coker\mathcal{F}_*$.
			
			The last part of the theorem follows easily from \ref{lemm:pi=aut}.
		\end{proof}
	
		\begin{rem}\label{rema:rev_univ_sii_alpha_iso}
			From \ref{theo:rev_reg_de_alex} it follows that $\widetilde{X}=\mathcal{F}\downarrow\ast$ is connected if and only if $\alpha$ is an epimorphism and that $p\colon \widetilde X \to X$ is the universal covering of $X$ if and only if $\alpha$ is an isomorphism.
		\end{rem}
		\begin{rem}\label{rema:rev_reg_de_alex}
			It is not hard to see that Theorem \ref{theo:rev_reg_de_alex} admits an analogous version with $\widetilde{X}=*\downarrow \mathcal{F}$. Indeed, the map $(x,g)\mapsto (x,g^{-1})$ from $\mathcal{F}\downarrow *$ to $*\downarrow\mathcal{F}$ is a homeomorphism over $X$.
		\end{rem}
		\begin{coro}\label{coro:rev_reg_de_alex}
			Let $X$ be a connected Alexandroff space, let $x_0\in X$ and let $H$ be a normal subgroup of $\pi_1(X,x_0)$. Then there exists a group $G$ and a functor $\mathcal{F}\colon X\to G$ such that the canonical projection $p\colon \mathcal{F}\downarrow *\to X$ is the connected regular covering of $X$ that corresponds to the subgroup $H$.
		\end{coro}
		\begin{proof}
			It follows immediately from Theorem \ref{theo:rev_reg_de_alex} by taking $G=\pi_1(X,x_0)/H$, the map $\alpha\colon \pi_1(X,x_0)\to G$ to be the quotient map and letting $\mathcal{F}=\mathcal{F}_{T,\alpha}$ for some maximal tree $T$ of $X$.
		\end{proof}
		\begin{ex}[{\cite[Example 2.4]{cianci2018splitting}}]\label{exam:revestimiento}
			Let $X$ be the following poset:
			\begin{displaymath}
				\begin{tikzpicture}[x=4cm,y=3cm]
					\tikzstyle{every node}=[font=\footnotesize]
					\draw (0,0) node[below]{$a$} node (a){$\bullet$};
					\draw (0.33,0) node[below]{$b$} node (b){$\bullet$};
					\draw (0,0.5) node[above]{$c$} node (c){$\bullet$};
					\draw (0.33,0.5) node[above]{$d$} node (d){$\bullet$};
					\draw (a)--(c);
					\draw (a)--(d);
					\draw (b)--(c);
					\draw (b)--(d);
				\end{tikzpicture}
			\end{displaymath}
			
			Let $T$ be the maximal tree of $X$ that contains the arrows $a\leq c$, $b\leq c$ and $b\leq d$, and let $\iota_X\colon X\to \loc X$ and $q_X\colon\loc X\to \loc X/\loc T$ as before. 

			We have that $\loc X/\loc T\cong\pi_1(X,a)\cong \Z$. It is not hard to see that
			$\aut_{\loc X}(a)$ is generated by the class of the sequence $(a\leq d,d\geq b,b\leq c,c\geq a)$, where $d\geq b$ and $c\geq a$ are the formal inverses of the arrows $b\leq d$ and $a\leq c$, respectively. It follows that $\eta=q_X(\iota_X(a\leq d))$ is a generator of the group $\loc X/\loc T$.
			
			Let $\alpha\colon\pi_1(X,a)\to\Z_6$ be the group homomorphism that maps the generator of $\pi_1(X,a)$ that identifies with $\eta$, to the element $2\in\Z_6$.
			The functor $\mathcal{F}_{T,\alpha}\colon X\to \Z_6$ is defined by $\mathcal{F}_{T,\alpha}(a\leq d)=2$, $\mathcal{F}_{T,\alpha}(a\leq c)=0$, $\mathcal{F}_{T,\alpha}(b\leq c)=0$ and $\mathcal{F}_{T,\alpha}(b\leq d)=0$.
			Hence, the space $\widetilde{X}=\mathcal{F}_{T,\alpha}\downarrow *$ is the following poset:
			
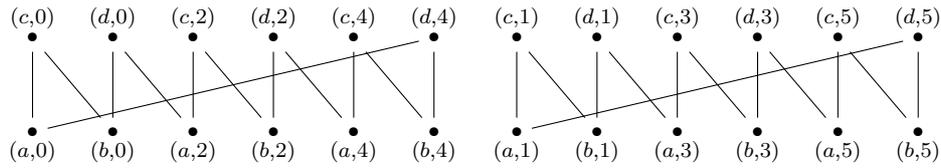
\begin{displaymath}
				\begin{tikzpicture}[x=3.2cm,y=2.5cm]
					\tikzstyle{every node}=[font=\footnotesize]
					\foreach \n in {0,2,4}
					\draw (0.33*\n-0.33,0.5) node[above]{($c$,\n)} node(c\n){$\bullet$};
					\foreach \n in {0,2,4}
					\draw (0.33*\n,0.5) node[above]{($d$,\n)} node(d\n){$\bullet$};
					\foreach \n in {0,2,4}
					\draw (0.33*\n-0.33,0) node[below]{($a$,\n)} node(a\n){$\bullet$};
					\foreach \n in {0,2,4}
					\draw (0.33*\n,0) node[below]{($b$,\n)} node(b\n){$\bullet$};
					\foreach \x in {0,2,4} \draw (a\x)--(c\x);
					\foreach \x in {0,2,4} \draw (b\x)--(d\x);
					\foreach \x in {0,2,4} \draw (b\x)--(c\x);
					\draw (a0)--(d4);
					\draw (a2)--(d0);
					\draw (a4)--(d2);
					\foreach \n in {1,3,5}
					\draw (1.66+0.33*\n-0.33,0.5) node[above]{($c$,\n)} node(c\n){$\bullet$};
					\foreach \n in {1,3,5}
					\draw (1.66+0.33*\n,0.5) node[above]{($d$,\n)} node(d\n){$\bullet$};
					\foreach \n in {1,3,5}
					\draw (1.66+0.33*\n-0.33,0) node[below]{($a$,\n)} node(a\n){$\bullet$};
					\foreach \n in {1,3,5}
					\draw (1.66+0.33*\n,0) node[below]{($b$,\n)} node(b\n){$\bullet$};
					\foreach \x in {1,3,5} \draw (a\x)--(c\x);
					\foreach \x in {1,3,5} \draw (b\x)--(d\x);
					\foreach \x in {1,3,5} \draw (b\x)--(c\x);
					\draw (a1)--(d5);
					\draw (a3)--(d1);
					\draw (a5)--(d3);
				\end{tikzpicture}
			\end{displaymath}
			and the projection $p\colon\widetilde{X}\to X$ is a covering map. It is clear that 
			\[\pi_0(\widetilde{X},(a,0))\cong \Z_2 \cong \Z_6/\im \alpha.\]
			Using \cite[Theorem 1]{mccord1966singular} and standard results about the fundamental group of $S^{1}$, it can be shown that 
			\[p_*(\pi_1(\widetilde{X},(a,0)))=\ker\alpha.\]
			
			Note that since $\mathcal{F}_{T,\alpha}$ is trivial on $\{a,b,c\}$, then $p^{-1}(\{a,b,c\})$ is nothing but the product $\{a,b,c\}\times \Z_6$, where $\Z_6$ is considered as a discrete topological space. On the other hand, $\mathcal{F}_{T,\alpha}$ is \emph{not} trivial on $\{a,b,d\}$ and therefore $p^{-1}(\{a,b,d\})$ is \emph{not} equal to the product $\{a,b,d\}\times \Z_6$. However, $p^{-1}(\{a,b,d\})$ is homeomorphic to the product $\{a,b,d\}\times \Z_6$ since there exists a maximal tree $T'$ of $X$ such that $\mathcal{F}_{T',\alpha}$ is trivial on $\{a,b,d\}$.
			Finally, it is obvious that $\widetilde{X}$ is not homeomorphic to the product $X\times\Z_6$ since there is no maximal tree $T''$ on $X$ such that $\mathcal{F}_{T'',\alpha}$ is the trivial functor.
		\end{ex}
		
		\bibliographystyle{acm}
		\bibliography{ref2}
	\end{document}